\documentclass[journal,onecolumn]{IEEEtran}

\usepackage{mathrsfs}
\usepackage{color}
\usepackage{enumerate}
\usepackage{mathrsfs}
\usepackage{amssymb}
\usepackage{amsfonts}
\usepackage{amsmath}
\usepackage{multirow}
\usepackage{latexsym}
\DeclareSymbolFont{AMSb}{U}{msb}{m}{n}
\DeclareMathSymbol{\subsetneq}{\mathrel}{AMSb}{"28}

\input cyracc.def

\def\mymedskip{\vskip\medskipamount}
\def\mymedbreak{\par \ifdim\lastskip<\medskipamount
  \removelastskip \penalty-100 \mymedskip \fi}
\def\myaftermedspace{\par \ifdim\lastskip<\medskipamount
  \removelastskip \penalty55\mymedskip\fi}
\newcommand{\eop}{{\unskip\nobreak\hfil\penalty50
          \hskip2em\hbox{}\nobreak\hfil$\Box$
          \parfillskip=0pt \finalhyphendemerits=0 \par}}
\newenvironment{proof}%
{\mymedbreak{\noindent\bf Proof:\enspace}}{\eop\myaftermedspace}
{\mymedbreak{\noindent\bf Proof of Theorem #1:\enspace}}{\eop\myaftermedspace}
\newtheorem{teor}{Theorem}[section]
\newtheorem{defi}[teor]{Definition}
\newtheorem{fact}[teor]{Fact}
\newtheorem{problem}{Problem}
\newtheorem{exercise}{Exercise}
\newtheorem{examp}[teor]{Example}
\newtheorem{exa}[teor]{Example}
\newtheorem{lem}[teor]{Lemma}
\newtheorem{cor}[teor]{Corollary}
\newtheorem{con}[teor]{Conjecture}
\newtheorem{prop}[teor]{Proposition}
\newtheorem{rem}[teor]{Remark}
\newtheorem{alg}[teor]{Algorithm}
\newcommand{\Tm}[1]{Theorem~\protect\ref{#1}}
\newcommand{\Le}[1]{Lemma~\protect\ref{#1}}
\newcommand{\Pn}[1]{Proposition~\protect\ref{#1}}

\newcommand{\Co}[1]{Corollary~\protect\ref{#1}}

\newcommand{\Sec}[1]{Section~\protect\ref{#1}}

\newcommand{\btm}[1]{\begin{teor} \label{#1}}
\newcommand{\etm}{\end{teor}}
\newcommand{\btmn}[2]{\begin{teor}[#1] \label{#2}}
\newcommand{\etmn}{\end{teor}}
\newcommand{\ble}[1]{\begin{lem} \label{#1}}
\newcommand{\ele}{\end{lem}}
\newcommand{\bLe}[1]{\begin{Lemma} \label{#1}}
\newcommand{\eLe}{\end{Lemma}}
\newcommand{\bpn}[1]{\begin{prop} \label{#1}}
\newcommand{\epn}{\end{prop}}
\newcommand{\bex}[1]{\begin{examp} \label{#1}}
\newcommand{\eex}{\hfill$\Box$\end{examp}}
\newcommand{\bde}[1]{\begin{defi} \label{#1}}
\newcommand{\ede}{\end{defi}}
\newcommand{\bco}[1]{\begin{cor} \label{#1}}
\newcommand{\eco}{\end{cor}}
\newcommand{\bcorn}[2]{\begin{cor}[#1] \label{#1}}
\newcommand{\ecorn}{\end{cor}}
\newcommand{\bcon}[1]{\begin{con} \label{#1}}
\newcommand{\econ}{\end{con}}
\newcommand{\bfact}[1]{\begin{fact} \label{#1}}
\newcommand{\efact}{\end{fact}}
\newcommand{\bpr}[1]{\begin{problem} \label{#1}}
\newcommand{\epr}{\end{problem}}
\newcommand{\bprnn}[1]{\begin{problemnn} \label{#1}}
\newcommand{\eprnn}{\end{problemnn}}
\newcommand{\bprn}[2]{\begin{problem}[#1] \label{#2}}
\newcommand{\eprn}{\end{problem}}
\newcommand{\bexer}[1]{\begin{exercise} \label{#1}}
\newcommand{\eexer}{\end{exercise}}
\newcommand{\bre}[1]{\begin{rem} \label{#1}}
\newcommand{\ere}{\end{rem}}
\newcommand{\balg}[1]{\begin{alg} \label{#1}}
\newcommand{\ealg}{\end{alg}}
\newcommand{\bqu}{\begin{question}}
\newcommand{\equ}{\end{question}}
\newcommand{\bs}{\begin{solution}}
\newcommand{\es}{\end{solution}}
\newcommand{\bh}{\begin{hint}}
\newcommand{\eh}{\end{hint}}
\newcommand{\bms}[1]{\begin{multisolution}{#1}}
\newcommand{\ems}{\end{multisolution}}
\newcommand{\bquo}{\begin{quote}}
\newcommand{\equo}{\end{quote}}
\newcommand{\beq}{\begin{equation}}
\newcommand{\eeq}{\end{equation}}
\newcommand{\beql}[1]{\begin{equation} \label{#1}}
\newcommand{\eeql}{\end{equation}}
\newcommand{\beqa}{\begin{eqnarray*}}
\newcommand{\eeqa}{\end{eqnarray*}}
\newcommand{\beqal}[1]{\begin{eqnarray} \label{#1}}
\newcommand{\eeqal}{\end{eqnarray}}
\newcommand{\beqan}{\begin{eqnarray}}
\newcommand{\eeqan}{\end{eqnarray}}
\newcommand{\bpf}{\begin{proof}}
\newcommand{\epf}{\end{proof}}

\newcommand{\cA}{{\cal A}}
\newcommand{\cB}{{\cal B}}

\newcommand{\cN}{{\cal N}}

\newcommand{\cZ}{{\cal Z}}

\newcommand{\bC}{{\bf C}}
\newcommand{\bF}{{\bf F}}

\newcommand{\bZ}{{\bf Z}}

\newcommand{\Aut}{{\rm Aut}}

\newcommand{\gs}{\sigma}

\newcommand{\ga}{\alpha}

\newcommand{\gd}{\delta}
\newcommand{\gl}{\lambda}

\newcommand{\bE}{{\bf E}}

\newcommand{\weight}{{\rm w}}
\newcommand{\supp}{{\rm supp}}
\newcommand{\hf}{\hat{f}}

\newcommand{\add}{+}

\newcommand{\lcm}{{\rm lcm}}
\newcommand{\expo}{{\rm exp}}
\newcommand{\chr}{{\rm char}}

\newcommand{\rnk}{{\rm rank}}

\begin{document}
\begin{titlepage}
\title{The shift bound for abelian codes and generalizations of the Donoho-Stark uncertainty principle}
\date{\today}
\author{%
Tao Feng\\
School of Mathematical\  Sciences\\
Zhejiang University\\
Hangzhou 310027\\
Zhejiang, China\\
{\tt tfeng@zju.edu.cn}%
\\\and
Henk D. L.\ Hollmann\\Philips IP\&S\\ HTC 34,
5656 AE Eindhoven, \\the Netherlands\\
{\tt henk.d.l.hollmann@philips.com}%
\\\and
Qing Xiang\\Department\  of Math.\
Sciences\\University of Delaware\\Newark, DE 19716, USA\\
{\tt qxiang@udel.edu}%
}
\maketitle
\thispagestyle{empty}
\begin{center}
Dedicated to the memory of J.H.~van~Lint.
\vspace{12pt}
\end{center}
%
\begin{abstract} Let $G$ be a finite abelian group. If $f: G\rightarrow \bC$  is a nonzero function with Fourier
transform $\hf$, the Donoho-Stark uncertainty principle states that $|\supp(f)||\supp(\hf)|\geq |G|$. The purpose of this paper is twofold. 
First, we present the shift bound for abelian codes with a streamlined proof. 
Second, we use the shifting technique to prove a generalization and a sharpening  
of the Donoho-Stark uncertainty principle. In particular, the sharpened uncertainty principle states, with notation above, that $|\supp(f)||\supp(\hf)|\geq |G|+|\supp(f)|-|H(\supp(f))|,$
where $H(\supp(f))$ is the stabilizer of $\supp(f)$ in $G$.
\end{abstract}
\end{titlepage}

\newcommand{\Ghat}{\hat{G}}
\newcommand{\Hhat}{{\cal K}}
\newcommand{\Khat}{{\cal L}} 
\newcommand{\chat}{\hat{c}}
\newcommand{\fhat}{\hat{f}}
\newcommand{\fh}{\hat{f}}
\newcommand{\Char}{\Gamma}

\newcommand{\norm}[1]{\|#1\|_\infty}
\section{\label{int}Introduction}
%
%

In a breakthrough paper \cite{lw}, Van Lint and Wilson developed a technique called
{\em shifting\/} to obtain lower bounds for the minimum distance of cyclic codes. The best bound obtainable by
this technique is called the {\em shift bound\/} (or, sometimes, the Van Lint-Wilson
bound), which depends only on the {\em defining zeros\/} of the cyclic code.
The shift bound can be difficult to compute, but
given the corresponding shifting steps, it is easy to verify correctness
of the bound. So it can be used as a {\em proof certificate\/} that the minimum
distance of a given code has (at least) a certain value.

The shifting technique for cyclic codes can be easily generalized to abelian codes, which are
ideals of a group algebra $\bF[G]$, with $G$ being a finite abelian group and $\bF$ some
finite field. If the field has characteristic 0 or if the characteristic of the field does not
divide $|G|$ then a Fourier-type transform can be
defined on $\bF[G]$,
and such a code can then be characterized in terms of the
vanishing of certain Fourier transform coefficients for all codewords.
The shift bound for abelian codes generalizes the ordinary shift bound and now only
depends on the set of coefficients that vanish for all codewords.
After Section~\ref{pre}, which contains some background on characters of
abelian groups, we present the shift bound for abelian codes in Section~\ref{shi} with a streamlined proof.
In Section~\ref{exa} we give some examples to illustrate applications of the shift bound. In \Sec{Salt} we present an alternative, simpler derivation of the shift bound, based on ideas from \cite{mes-pq}.

The Donoho-Stark uncertainty principle for finite abelian groups states: If $G$ is a finite abelian group, and $f:  G\rightarrow \bC$  is a
nonzero complex-valued function with Fourier transform $\hf$, then $|\supp(f)||\supp(\hf)| \geq |G|$; equality holds if and only if
$f$ is a nonzero multiple of the restriction of a character to a coset of a subgroup of~$G$ (see, e.g., \cite{mat}, \cite{smi}).

It seems natural to try to generalize this principle to all fields for which a
Fourier-type transform exists. The original proof in \cite{mat} as well as the simple
proofs in \cite{qpv} or \cite{tao}
crucially depend on the existence of an absolute value, and hence do not generalize to
finite fields. It is not too difficult to see that the elementary induction proof in
\cite{pre2} does generalize (indeed, note that for cyclic groups, the principle can be
seen to follow from the BCH bound); however, the resulting proof is still rather complicated.
Since the shift bound for abelian codes provides a lower bound on the
weight $w(f)=|\supp(f)|$ of 
a nonzero function $f: G\rightarrow \bF$ in terms of the support of its Fourier transform $\hf$,
it seems reasonable to investigate whether the Donoho-Stark uncertainty principle
can be obtained as a consequence of the shift bound.
In Section~\ref{gen} we show that this is indeed the case. Here, we use the shift bound to derive a generalization of the Donoho-Stark uncertainty principle. 
We note that a similar approach was used in \cite{mes-pq} for a generalization to non-abelian groups.

In Section~\ref{fext} we use the shifting technique to prove a sharpening of the Donoho-Stark uncertainty principle. Let $G$ be a finite abelian group, $\bF$ a field of characteristic 0 or characteristic $p$ with $p\not |\,|G|$, 
and let $f: G\rightarrow \bF$ be a nonzero function with Fourier transform $\hf$. We obtain a pair of inequalities which are stronger than the Donoho-Stark uncertainty principle: 
$|\supp(f)||\supp(\hf)|\geq |G|+|\supp(\hf)|-|H(\supp(\hf))|$ and $|\supp(f)||\supp(\hf)|\geq |G|+|\supp(f)|-|H(\supp(f))|$, where $H(\supp(\hf))$ is the stabilizer of $\supp(\hf)$ in $\Ghat$, and $H(\supp(f))$ is similarly defined as a subset of $G$.

\section{\label{pre}Preliminaries}
In this section we summarize some of the theory of $\bE$-valued characters and Fourier
transforms over finite abelian groups, for general (possibly finite) fields $\bE$.
Readers who are familiar with Fourier theory might skip this section on first reading.
Further background on harmonic analysis and Fourier analysis on groups can be found, e.g., in
\cite{hew}, \cite{man}, \cite{rud}, \cite{ter}.
For the use of characters and Fourier transforms in relation to coding theory,
see, e.g., \cite{bla}, \cite{cam}, \cite{for}, where most of the results below can be
found. Abelian codes were first investigated in \cite{ber1}, \cite{ber2}.

Let $(G,+)$ be a finite abelian group.
We write $|G|$ to denote the {\em order\/} of $G$. The {\em exponent\/} $\expo(G)$ of
$G$ is defined as the smallest positive integer $N$ for which $Nx=0$ for all $x\in G$,
where $0$ denotes the identity element of $(G,+)$.

In the remainder of this paper, $\bF$ is a field of characteristic $\chr(\bF)=0$ or
$\chr(\bF)=p$ with $p \not |\,|G|$, and $\bE$ is an extension of $\bF$ containing a
primitive $N$-th root of unity $\xi$, an element of multiplicative order $N$ in
$\bE$, where $N=\expo(G)$. Note that our assumption on $\chr(\bF)$ is necessary and
sufficient to guarantee that such an extension exists.
A character $\chi$ of $G$ is a homomorphism of $(G,+)$
to the cyclic group of order $N$ generated by $\xi$, and hence takes its values in
the field~$\bE$. We will refer to such a character as a {\em $\bE$-valued character\/}.
%
These characters form a group
$(\Ghat,+)$ under the operation of pointwise multiplication defined by
\[ (\chi+\phi)(g) = \chi(g)\phi(g) \]
for $\chi, \phi \in \Ghat$ and $g\in G$.

The abelian group $G$ is isomorphic to a direct product
\[ G \cong \bZ_{n_1} \times \cdots \times \bZ_{n_r}\]
of cyclic groups; note that
\[ N=\expo(G) = \lcm(n_1, \ldots, n_r).\]
For each $e=(e_1, \ldots, e_r)\in G$, define the map $\chi_e: G \rightarrow \bE$ by
\[\chi_e(x) = \xi^{\sum_{i=1}^r \frac{N}{n_i}e_i x_i}, \]
for all $x=(x_1, \ldots, x_r)\in G$. It is easy to see that
each $\chi_e$ is a character, and $\chi_a+\chi_b=\chi_{a+b}$ for all
$a,b\in G$. Moreover, since $\xi$ is a primitive $N$-th root of unity, all the $\chi_e$'s
are {\em distinct\/} and, in fact, it is easily shown that
each character is of this form.
Hence the group $(\Ghat,+)$ of characters of $G$ is isomorphic to $(G,+)$. Note that the identity element of $\Ghat$ is $\chi_0: x\mapsto 1$ for all $x\in G$. An important property is that for every $x\in G\setminus \{0\}$ there exists a character $\chi\in \Ghat$ for which $\chi(x)\neq 1$.

For $a\in G$, define $\Phi_a: \Ghat \rightarrow \bE$ by letting 
\[\Phi_a(\chi)=\chi(a)\]
for all $\chi\in \Ghat$. Note that $\Phi_a$ is a character on $\Ghat$, that is, an element of $\hat{\Ghat}$. In fact, it turns out that $G$ and $\hat{\Ghat}$ are isomorphic (Pontryagin duality), with the map $a\rightarrow \Phi_a$ being an isomorphism.

Given the group $(\Ghat,+)$ of $\bE$-valued characters and a function $f: G \rightarrow \bE$, we define the
{\em Fourier transform\/} $\hf$ of $f$ by
\[ \hf(\chi) = \sum_{x\in G} f(x) \chi(-x), \]
for all $\chi \in \Ghat$.
The {\em supports\/} $\supp(f)$ and $\supp(\hf)$ of $f$ and $\hf$ are defined respectively by
\[ \supp(f) = \{ x \in G \mid f(x) \neq 0\}, \qquad
\supp(\hf) = \{ \chi \in \Ghat \mid \hf(\chi) \neq 0\}. \]

The {\em group algebra\/} $\bF[G]$ consists of all formal sums
\[ f = \sum_{x\in G} f(x) x, \]
with $f(x) \in \bF$. In what follows, we will not distinguish between the element $f$
in~$\bF[G]$ written as a vector
\[ f = (f(x_1), \ldots, f(x_n)) \]
in $\bF^n$, where $G=\{x_1, \ldots, x_n\}$,
and the function $f: G \rightarrow \bF$ given by
\[ f : x \mapsto f(x),\]
for all $x\in G$.
Addition and scalar multiplication in $\bF[G]$ are defined by the corresponding vector
operations, and multiplication $*$ in $\bF[G]$ is the convolution operation defined by
\[ (f*g)(z) =\sum_{x\in G} f(x)g(z-x)\]
for all $z\in G$.

%
The characters $\chi\in\Ghat$, considered as elements in $\bE[G]$, constitute a basis
of $\bE[G]$; the Fourier transform can then be understood in terms of a base
change since
\[ f = \sum_{x\in G} f(x) x = |G|^{-1}\sum_{\chi \in \Ghat} \hf(\chi) \chi \]
holds for all $f \in \bE[G]$.
In fact, the characters constitute an orthogonal basis of eigenfunctions, with 
\[ \chi*\chi' = 
\left\{
\begin{array}{cc}  |G| \chi, & \mbox{if $\chi=\chi'$;}\\
		   0  , & \mbox{otherwise},
\end{array}
\right.
\]
and 
\[ f*\chi=\hf(\chi) \chi\]
for all $\chi, \chi'\in \Ghat$ and $f\in \bE]G]$.
   

For our investigation of the case of equality in the generalized Donoho-Stark principle
in Section~\ref{shi}, we need some additional facts.
Firstly, the {\it Inverse Fourier Transform} for functions $g: \Ghat \rightarrow \bE$ defined by
\[ g^*(x) = |G|^{-1} \sum_{\chi\in \Ghat} g(\chi)\chi(x) \]
for $x\in G$ is the inverse of the Fourier transform, that is, $(\hf)^*=f$ for
all functions $f: G \rightarrow \bE$.

Next, let $H$ and $\Hhat$ be subgroups of $G$ and $\Ghat$, respectively. We define
$H^\perp$ and $\Hhat^\perp$ by
\[ H^\perp = \{ \chi \in \Ghat \mid \mbox{$\chi(y) = 1$ for all $y\in H$} \} \]
and
\[ \Hhat^\perp = \{ x \in G \mid \mbox{$\eta(x) = 1$ for all $\eta \in \Hhat$} \}. \]
Then $H^\perp$ and $\Hhat^\perp$ are subgroups of $\Ghat$ and $G$, respectively, with
\beql{LEorth} |H^\perp| = |G|/|H|, \qquad |\Hhat^\perp| =  |\Ghat|/ |\Hhat|. \eeql
Moreover, if $H=\Hhat^\perp$, then $H^\perp = \Hhat$.
Finally, we have that
\[ \sum_{y \in H} \chi(y) =
\left\{
\begin{array}{cc}  |H|, & \mbox{if $\chi\in H^\perp$;}\\
		   0  , & \mbox{otherwise},
\end{array}
\right.
\]
and  
\[ \sum_{\chi \in H^\perp} \chi(y) =
\left\{
\begin{array}{cc}  |H^\perp|, & \mbox{if $y\in H$;}\\
		   0  , & \mbox{otherwise}.
\end{array}
\right.
\]

Using the above facts it is not difficult to show the following.
\begin{teor}\label{tsup} Let the function $f: G\rightarrow \bE$ be such that its Fourier
Transform $\hf$
has support equal to a coset $\phi+\Hhat$ of a subgroup $Hhat$ of $\Ghat$. 
Then $|\supp(f)| |\supp(\fh)|=|\Ghat|$ holds if and only if $f$ and $\hf$ are of the
form
\[ f=\lambda \phi I_{a+H}, \qquad \fhat=\mu\Phi_{-a} I_{\phi+\Hhat},\]
for some $a\in G$, with $\lambda=\phi(-a)f(a)\in\bE\setminus\{0\}$ and $\mu=|\Hhat|f(a)$;
here $H=\Hhat^\perp$,
$\Phi_{-a}: \chi\mapsto \chi(-a)$ is the character in $\hat{\Ghat}\cong G$ associated
with $-a$, and $I_{a+H}$ and
$I_{\phi+\Hhat}$ denote the indicator functions of $a+H$ and $\phi+\Hhat$,
respectively.
\end{teor}
\begin{proof} By the inversion formula, we have 
$f(x)=|G|^{-1}\sum_{\chi\in\Ghat} \hf(\chi)\chi(x)$.
If $\fhat(\chi)=0$ for all $\chi\notin \phi+\Hhat$, then
for any $x\in G$ and $h\in H$, we have 
\beqa
f(x+h)&=&|G|^{-1}\sum_{\chi\in\phi+\Hhat} \fhat(\chi)\chi(x+h)\\
	&=& |G|^{-1}\phi(h) \sum_{\chi\in\phi+\Hhat} \fhat(\chi) \chi(x)\\
	&=& \phi(h)f(x),
\eeqa
where we used that $H=\Hhat^\perp$ and $\chi(h)=\phi(h)$ for $\chi\in\phi+\Hhat$ and $h \in \Hhat^\perp$.
Since $\phi(x)\neq0$ for all $x\in G$, we conclude from the above that the support
of $f$ is a union of cosets of $H$. Since $|H|=|\Hhat^\perp|=|\Ghat|/|\Hhat|$, we have $|\supp(f)||\supp(\hf)|=|G|$ only if the support of $f$ is $a+H$, for some $a\in G$, so that $f$
is of the form
$f=\lambda \phi I_{a+\Hhat}$ with $\lambda=\phi(-a)f(a)$.
Moreover, given that $f$ is of this form, using the fact that $H^\perp=\Hhat$,  we obtain that
$\fhat=\mu \Phi_{-a} I_{\phi+\Hhat}$ with $\mu=|H|f(a)$, as claimed.
\end{proof}
\section{\label{shi}Shifting for abelian codes}
We now discuss a technique called {\em shifting\/} to find lower bounds on the
minimum weight of abelian codes.
This technique is a straightforward generalization of the shifting technique
introduced in \cite{lw} to obtain lower bounds on the minimum distance of cyclic codes.

As before, $(G,+)$ is a finite abelian group, $\bF$ is a field with $\chr(\bF)=0$
or $\chr(\bF)=p$ with $p\not |\, |G|$, and
$\bE\supseteq \bF$ is a field extension of $\bF$ containing a primitive
$N$-th root of unity, where $N$ denotes the exponent ${\rm exp}(G)$ of $G$.
Recall that the collection of $\bE$-valued characters on $G$ forms a group $(\Ghat,+)$
isomorphic to $(G,+)$.

Now let $f: G \rightarrow \bF$ be an $\bF$-valued function with $\hf$ as its Fourier transform.
We define
\beql{Zdef}
\cZ (f)= \{\chi \in \Ghat \mid \hf(\chi) = 0\},
\qquad \cN (f) = \Ghat \setminus \cZ. \eeql
We call $\cZ(f)$ and $\cN(f)$ the
{\em zeros and nonzeros of $f$\/} (in $\Ghat$), respectively. Note that $\supp(\hf)=\cN(f)$.

Let $\cZ\subseteq\Ghat$. The ideal $C$ in $\bF[G]$ consisting of all $f\in \bF[G]$ whose zeros include $\cZ$, i.e., $$C=\{f\in \bF[G]\mid \cZ(f)\supseteq \cZ\},$$  is called the {\em abelian code with $\cZ$ as defining zeros\/}.
Note that if $G$ is cyclic of order $n$, then $\Ghat$ essentially is the collection
$\bE_n$ of $n$-th roots of unity in $\bE$ and $C$ is just the cyclic code with
defining zeros $\cZ\subseteq \bE_n$.

Any field automorphism $\sigma: \bE\rightarrow \bE$ of $\bE$ that fixes $\bF$ pointwise (that is, $\gs\in \Aut(\bE/\bF)$)
induces a map on $\Ghat$ (which we denote again by $\sigma$) defined by
$\sigma: \chi \mapsto \chi^\sigma$, where $\chi^\sigma(x) = \chi(x)^\sigma$ for $x\in G$.
A subset of $\Ghat$ that is closed under {\em all\/} field automorphisms in $\Aut(\bE/\bF)$ will be called {\em $\bF$-closed\/}.
Note that the set of zeros $\cZ(f)$ of an $\bF$-valued function $f: G \rightarrow \bF$ 
is $\bF$-closed. Similarly, if the ideal $C$ in $\bF[G]$ has the set $\cZ$ as defining
zeros, then the collection $\cZ'$ of common zeros of elements of $C$, called the
{\em complete\/} set of zeros of $C$, is just the {\em $\bF$-closure\/} of $\cZ$,
the smallest $\bF$-closed superset of $\cZ$.

Now let $f$ be a nonzero function in $\bF[G]$, and let $\cZ=\cZ(f)$. Assume that $f$ have support ${\rm supp}(f)=S$, where
\[ {\rm supp}(f)=\{x\in G \mid f(x)\neq 0\}. \]
Write $S=\{x_1, \ldots, x_w\}$, where $w=\weight(f)=|{\rm supp}(f)|$ is the
{\em weight\/} of the vector $f$.
With each $\chi \in \Ghat$ we associate a vector $v(\chi)$ in $\bE^w$ defined by
\[ v(\chi)=(\chi(-x_1), \ldots, \chi(-x_w))^\top; \]
also, we define $\gamma = \gamma(f)$ in $\bF^w$ by
\[ \gamma=(f(x_1), \ldots, f(x_w))^\top. \]
As a consequence of these definitions, we have $\chi\in \cZ$ if
and only if $\gamma \perp v(\chi)$.
Finally, for $\psi \in \Ghat$, write $D(\psi)$ to denote the diagonal matrix
\[ D(\psi)={\rm diag}(\psi(-x_1), \ldots, \psi(-x_w)). \]
Note that
\[D(\psi)v(\chi)=v(\psi\add\chi),\]
where $\psi\add\chi$ is the character in $\Ghat$ defined by
$(\psi\add\chi)(x)=\psi(x)\chi(x)$ for all $x\in G$.

We say that the set $\cA\subseteq \Ghat$ is {\em independent\/} if the
corresponding set of vectors $V(\cA)=\{v(\chi)\mid \chi \in \cA\}$ is independent
in~$\bE^w$. Our interest in independent subsets of $\Ghat$ stems from the fact that if
$\cA\subseteq \Ghat$ is independent, then $\weight(f)=w\geq |\cA|$.
The next lemma, which is the key result for the shift bound, provides a means to
{\em construct\/} independent sets in $\Ghat$.
\begin{lem}
\label{Zindep}
\mbox{}
\begin{enumerate}
\item {}
[initialize] $\emptyset$ is independent;
\item {}[shifting] If $\cA$ is independent and if $\psi \in \Ghat$, then
$\psi\add\cA=\{\psi\add\chi \mid \chi \in \cA\}$ is independent;
\item {}[extension] If $\cA\subseteq \cZ$ is independent and if $\eta \notin \cZ$,
then $\cA\cup \{\eta\}$ is independent;
\item {}[field automorphisms] If $\sigma\in \Aut(\bE/\bF)$, then
$\sigma(\cA)=\{\chi^{\sigma} \mid \chi \in \cA\}\subseteq \cZ$, and $\sigma(\cA)$ is independent.
\end{enumerate}
\end{lem}
\begin{proof}\\
1. Evident.\\
2. Since $\psi(x)\neq 0$ for all $x \in G$, the diagonal matrix $D(\psi)$ is nonsingular.
Since $D(\psi)V(\cA) = V(\psi\add\cA)$, the result follows.\\
3. Since $\cA\subseteq \cZ$ and $\eta\notin\cZ$, we see that the vector $\gamma$ is orthogonal to
all vectors in~$V(\cA)$ and $\gamma=\gamma(f)$ is not orthogonal to $v(\eta)$; hence $v(\chi)$
cannot be contained in the linear span of $V(\cA)$.\\
4. Evident.
\end{proof}

The rules 1--3 in Lemma~\ref{Zindep} inductively define a family of independent subsets of
$\Ghat$ that only depend on the subset $\cZ=\cZ(f)$ of $\Ghat$. We will call such sets
{\em independent with respect to~$\cZ$\/} or, more briefly, {$\cZ$-independent\/}.
%
\Le{Zindep} has the following immediate consequence.
\begin{teor}\label{shift}
Let $f: G \rightarrow \bF$ be a nonzero function from an abelian group $(G,+)$ to some field
$\bF$ of characteristic zero or of characteristic $p$ relatively prime to $|G|$, $\cZ=\{\chi \in \Ghat\mid \hf(\chi)=0\}$ be the set of zeros of $f$, and let $\weight(f)=|\supp(f)|$. Then 
\[ \weight(f) \geq |\cA| \]
for every $\cZ$-independent subset $\cA$ of $\Ghat$.
\end{teor}

For a subset $\cZ$ of an abelian group $\Ghat$, we denote by $\gd(\Ghat, \cZ)$ the
largest size of a $\cZ$-independent subset of $\Ghat$. 
Then Theorem~\ref{shift} has the following consequence.
\begin{teor}[the shift bound for abelian codes] \label{code-shift}
Let $(G,+)$ be an abelian group and let $\bF$ be a field of $\chr(\bF)=0$ or $\chr(\bF)=p$ not dividing $|G|$.
If $C$ is an abelian code in $\bF[G]$ with set of defining zeros $\cZ$,
then the minimum weight $d(C)$ of the code $C$ satisfies
\[ d(C) \geq \min \gd(\Ghat, \cZ'),\]
where the minimum is over all $\bF$-closed proper subsets $\cZ'$ of $\Ghat$ such that
$\cZ'\supseteq \cZ$.
\end{teor}
%
%
%

\section{\label{exa}Some examples of shifting}
In this section we illustrate the shifting method by discussing a couple of applications.
Readers who are mainly interested in the Donoho-Stark uncertainty principle can skip
this section.
\begin{exa} \rm [The BCH-bound for abelian codes] Let $C\subseteq \bF[G]$ be an abelian code with
defining zeros $\cZ$ in $\Ghat$, where $\bF$ is a field of characteristic $p$ not dividing $|G|$. If there is a character $\chi\in\Ghat$ 
and integers $d, a$ such that $\cZ$ contains all zeros 
$\chi^i$ for $a \leq i \leq a+d-2$ and if $D$ is the abelian code with defining zeros
\[ \cZ \cup \{ \chi^i \mid i\geq 0\}, \]
then the minimum weight $d(C)$ of $C$ satisfies
\[ d(C) \geq \min( d, d(D)). \] 

To see this, consider a word $c$ in $C$; then either $c$ is contained in $D$ and
$w(c)\geq d(D)$, or there is a $e\geq d$ such that all $\chi^i$ for $a \leq i \leq
a+e-2$ are zeros of $c$ but $\chi^{a+e-1}$ is a nonzero of $c$. In the latter case,
we can use the shifting rules from Lemma~\ref{Zindep} to construct independent sets
as follows:
\beqa \emptyset  \mapsto \emptyset \cup\{\chi^{a+e-1}\}  
	 \mapsto \{\chi^{a+e-2}\} &\cup& \{\chi^{a+e-1}\} \\
	 &\mapsto& \ldots 
	 \mapsto \{\chi^{a}, \ldots, \chi^{a+e-2}\} \cup \{\chi^{a+e-1}\};
\eeqa
hence the set $\{\chi^{a}, \ldots, \chi^{a+e-1}\}$ is independent, of size $e\geq d$,
so that by Theorem~\ref{shift} we have that $w(c)\geq d$.

An application of this bound can be found for example in \cite{dgv}.
\end{exa}

\begin{exa} \rm 
Let $\bF=\bF_2$, and consider the binary abelian code $C$ over $G=\bZ_7\times \bZ_7$
with defining zeros 
\[ (0,0), (0,1), (0,3), (1,0), (3,0), (1,1), (1,2), (1,4), (3,3), (3,5), (3,6)\]
in the dual group $\Ghat=\bZ_7\times\bZ_7$.
This code is not equivalent to a cyclic code.
Note that if $(x,y)$ is a zero, then $(2x,2y)$ and $(4x,4y)$ are also (conjugate) 
zeros. So the full set of zeros of codewords is 
\beqa \cZ = \{\hspace{-0.25in}& &(0,0), (0,1), (0,2), (0,3), (0,4), (0,5), (0,6),
(1,0), (2,0), (3,0), (4,0), (5,0), (6,0), \\
&&(1,1), (1,2), (1,4), (2,1), (2 ,2), (2,4), (4,1), (4,2), (4,4), \\
&&(3,3), (3,5), (3,6), (5,3), (5,5), (5,6), (6,3), (6,5), (6,6)\}.
\eeqa
Note that this code is the collection of all polynomials 
\[ c(x,y) = \sum_{i=0}^6\sum_{j=0}^6 c_{i,j} x^i y^j\]
for which $c(\ga^i, \ga^j)=0$ for all pairs $(i,j)$ in the above list,
where $\ga$ is primitive in $\bE=\bF_8$.
Possible nonzeros of codewords are $(1,3), (1,6), (1,5)$, their conjugates, and their
symmetric counterparts $(3,1), (6,1), (5,1)$ and their conjugates.

This code has length $n=49$, dimension $k=18$, and minimum distance $d=12$.
Note that in this case a BCH bound can be at most seven since each 
non-identity element of $\bE^*$ has order seven. 
To prove that the minimum distance $d$ of the code satisfies $d\geq 12$ with shifting, 
first we assume that $(1,3)$, and hence also 
$(2,6)$ and $(4,5)$, are nonzeros of a codeword $c$. Then shift as follows:
\beqa &&\emptyset  \mapsto \emptyset\cup\{(1,3)\}  
	 \mapsto \{(0,2)\} \cup \{(1,3)\}  
	\mapsto \{(0,0), (1,1)\} \cup \{(1,3)\}  \mapsto\\
	&& \{(6,5), (0,6), (0,1)\} \cup \{(2,6)\} \mapsto
	\{(6,0), (0,1), (0,3), (2,1)\} \cup \{(4,5)\}  \mapsto \\
	&&\{(6,3), (0,4), (0,6), (2,4), (4,1)\} \cup \{(2,6)\} \mapsto\\ 
	&& \{(4,4), (5,5), (5,0), (0,5), (2,2), (0,0)\} \cup \{(2,6)\} \mapsto \\
	& & \{(6,6), (0,0), (0,2), (2,0), (4,4), (2,2), (4,1)\} \cup \{(4,5)\} \mapsto
 \\
	& & \{(4,2), (5,3), (5,5), (0,3), (2,0), (0,5), (2,4), (2,1)\} \cup \{(2,6)\} \mapsto  \\
	& & \{(2,4), (3,5), (3,0), (5,5), (0,2), (5,0), (0,6), (0,3), (0,1)\} \cup \{(1,3)\}  \mapsto \\
	& & \{(1,1), (2,2), (2,4), (4,2), (6,6), (4,4), (6,3), (6,0), (6,5), (0,0)\} \cup \{(1,3)\}  \mapsto \\
	& & \{(2,2), (3,3), (3,5), (5,3), (0,0), (5,5), (0,4), (0,1), (0,6), (1,1),
(2,4)\} \cup \{(2,6)\}, 
\eeqa
thus proving that $w(c)\geq 12$ in this case.

So we may assume that $(1,3)$, $(2,6)$, $(4,5)$ are also zeros of~$c$.
Now assume that $(1,5)$, and hence also $(2,3)$ and $(4,6)$, are nonzeros of~$c$. Then 
shifting proves that first the sets
$$\emptyset, \{(1,5)\},
\{(1,4),(1,5)\}, \ldots, \{(1,0), (1,1), \ldots,  (1,5)\},$$ 
and then the sets
$$\{(0,0), (0,1), \ldots, (0,5), (1,5)\}, \ldots, \{(0,2), (0,3), \ldots,(0,6), (0,0),
(1,0), \ldots, (1,5)\}$$ 
are all independent, thus again proving that $w(c)\geq 12$ in this case.

So we may assume that $(1,5)$, $(2,3)$ and $(4,6)$ are also zeros of~$c$. Now assume that 
$(1,6)$, and hence also $(2,5)$ and $(4,3)$, are nonzeros of~$c$. Then
a similar shifting procedure proves first that the sets 
$$\emptyset, \{(1,6)\},
\{(1,5),(1,6)\}, \ldots, \{(1,0), (1,1), \ldots,  (1,6)\},$$ 
and then the sets
$$\{(0,0), (0,1), \ldots, (0,6), (1,6)\}, \ldots, \{(0,1), (0,2), \ldots,(0,6), (0,0),
(1,0), \ldots, (1,6)\}$$ 
are all independent,  thus proving that now $w(c)\geq 14$.
Now use the fact that the original zero set is symmetric under $(x,y)\mapsto (y,x)$ as
follows: if one of $(1,3)$ or $(3,1)$ is a nonzero, then $w(c)\geq12$; otherwise both 
$(1,3)$ and $(3,1)$ (and all their conjugates) are zeros. In that case, if one of
$(1,5)$ or $(5,1)$ is a nonzero, then again $w(c)\geq12$; otherwise also both $(1,5)$ and
$(5,1)$ (and all their conjugates) are zeros. Finally, in that case, if one of
$(1,6)$ or $(6,1)$ is a nonzero, then $w(c)\geq 16$; otherwise also both $(1,6)$ and
$(6,1)$ (and all their conjugates) are zeros. But then all elements of $\Ghat$ are
zeros, and the codeword is the all-zero word.

This code was investigated in \cite{cam}, where it was shown that the distance is 12 by other means.
\end{exa}

\section{\label{Salt}An alternative derivation of the shift bound}
In this section we relate the shift bound to a method from \cite{mes-pq}. We begin by recalling some notions from that paper.
Let $G$ be a finite abelian group, $f: G\rightarrow \bF$ be a nonzero $\bF$-valued function, and let $S=\supp(f)$ be the support of $f$. As we did before, we identify $f$ with the element $f=\sum_{x\in G} f(x)x$ in the group algebra $\bF[G]$. Define a linear map $T_f: \bF[G]\rightarrow \bF[G]$ by $T_f(u)=f* u$ for all $u\in \bF[G]$. Since $T_f(\chi)=\hf(\chi) \chi$ for every character $\chi\in \Ghat$, we see that the rank of $T_f$ equals $|\supp(\hf)|$.
Now suppose that $x_1, \ldots, x_t\in G$ have the property that $x_i+S\not\subseteq (x_1+S)\cup (x_2+S)\cup \cdots \cup (x_{i-1}+S)$ for $i=2, 3 \ldots, t$, where $x+S=\{x+s\mid s\in S\}$. Then from this property of the support $S$ and the fact that $T_f(x_i)=\sum_{z\in x_i+S}f(z-x_i)z$,  it follows that no linear combinations $\gl_1 T_f(x_1)+\cdots+\gl_tT_f(x_t)$, with $\gl_i\in \bF$ for all $i$, can be 0 unless $\gl_t=\cdots=\gl_1=0$; that is,  $T_f(x_1), \ldots,T_f(x_t)$ are independent in $\bF[G]$, and hence $\rnk(T_f)=|\supp(\hf)|\geq t$.
We can formalize the above as follows.
\bde{LDmes} Let $(G,+)$ be an abelian group, and let $S\subseteq G$ be a nonempty subset of~$G$. We say that the 
sequence $x_1, \ldots, x_t$ in $G$ has {\em $S$-rank\/} $t$  in~$G$  if 
$x_i+S\not\subseteq (x_1+S)\cup (x_2+S)\cup \cdots \cup (x_{i-1}+S)$ for $i=2, 3, \ldots, t$.
\ede
Then the discussion preceding the above definition  can be stated as follows.
\bpn{LPrb}\rm
Let $f: G\rightarrow \bF$ be nonzero and let $S=\supp(f)$. If there exists a 
sequence in~$G$ 
with $S$-rank~$t$, then $|\supp(\hf)|\geq t$.
\epn
By dualizing, we obtain the following.
\bco{LCrb}\rm
Let $G$ be an abelian group, $\bF$ be a field of characteristic 0 or with ${\rm char}(\bF)$ not dividing $|G|$, $f: G\rightarrow \bF$ be a nonzero function, and let $\cN=\supp(\hf)$. If  there exists a sequence in $\Ghat$ with $\cN$-rank $t$, then $|\supp(f)|\geq t$.
\eco
\bpf
Immediate consequence of the fact that $\hat{\Ghat}=G$ and $(\hat{f})^*=f$.
\epf
%
%
We now show that the lower bound on the minimum distance of an abelian code afforded by \Co{LCrb} is equivalent to the shift bound, an observation that seems to be new. We need some preparation. 
\bde{LDsbs}\rm 
Let $(G,+)$ be an abelian group and let $\cZ\subsetneq \Ghat$.
We say that the sequence $\ga_1, \ldots, \ga_t$  in $\Ghat$ is {\em $\cZ$-independent\/} in~$\Ghat$ if 
there are $\psi_1, \psi_2, \ldots, \psi_t$ in $\Ghat$ such that $\psi_1+\ga_1\notin\cZ$ and for $2\leq i\leq t$, we have $\psi_i+\ga_i\notin \cZ$ and $\psi_i+\ga_1, \psi_i+\ga_2, \ldots, \psi_i+\ga_{i-1} \in \cZ$.
\ede
%
\bpn{LPAind}\rm Let $\cZ\subsetneq\Ghat$ be a proper subset of $\Ghat$.
A set $\cA\subseteq \Ghat$ is $\cZ$-independent if and only if,  for some ordering $\cA=\{\ga_1, \ldots, \ga_t\}$ of the elements of $\cA$, the sequence $\ga_1, \ldots, \ga_t$ is $\cZ$-independent in $\Ghat$.
\epn
\bpf
Since $\cN\neq \emptyset$, the statement holds for $t=1$. Since the only way to enlarge independent sets is through rule 2, we see that a set $\cA$ of size $t\geq2$ is $\cZ$-independent if and only if it can be written as $\cA=\cA'\cup\{\ga_t\}$ with $\cA'$ being $\cZ$-independent and with $\psi_t+\cA'\subseteq \cZ$ and $\psi_t+\ga_t\in \cN$ for some $\psi_t\in \Ghat$. Now the statement follows by induction on $t$.
\epf

\bpn{LLmes}\rm Let $\cZ\subsetneq \Ghat$, and let $\cN=\Ghat\setminus \cZ$. The sequence $ \ga_1, \ldots, \ga_t$  has $\cN$-rank $t$ in $\Ghat$ if and only if the sequence $-\ga_1, \ldots, -\ga_t$ is $\cZ$-independent.
\epn
\bpf
Write $\cZ=\Ghat\setminus \cN$.
By definition, a sequence $\ga_1, \ldots, \ga_t$ has $\cN$-rank $t$ in $\Ghat$ precisely when $\ga_i+\cN\not\subseteq \ga_1+\cN\cup \cdots \cup \ga_{i-1}+\cN$ for $i=2, \ldots, \cN$, or, equivalently, if there exist $\psi_2, \ldots, \psi_t$ such that $\psi_i\in \ga_i+\cN$ and $\psi_i\in (\ga_1+\cN\cup \cdots \cup \ga_{i-1}+\cN)^c=\ga_1+\cZ\cap \cdots \cap \ga_{i-1}+\cZ$ for $i=2, \ldots, t$. Since $\cN\neq \emptyset$ by assumption, the claim now follows from \Pn{LPAind}.
\epf
As a consequence of \Pn{LPAind} and~\ref{LLmes}, we have the following alternative description of the shift bound.
\btm{LTmes}\rm
Let $(G,+)$ be abelian and let $\cZ\subset \Ghat$ be a proper subset of $\Ghat$. Put $\cN=\Ghat\setminus \cZ$. Then $\gd(\Ghat,\cZ)$ is equal to the largest integer $t$ for which there exists a sequence of $\cN$-rank $t$ in $\Ghat$.
\etm
In view of this theorem, \Co{LCrb} provides an alternative derivation of the shift bound for abelian codes in \Tm{code-shift}.
\section{\label{gen}A generalization of the Donoho-Stark uncertainty principle}
The Donoho-Stark uncertainty principle for finite abelian groups states: If $f:
G\rightarrow \bC$ is a nonzero complex function on a finite abelian group $(G,+)$ with $\hf$ its
complex Fourier transform, then
$$|\supp(f)| |\supp(\hf)| \geq |G|;$$
equality holds
if and only if $f$ takes the form
\[ f = c \chi I_{a+H} \]
for some $c\in\bC^*$ and some complex character $\chi$, where
$I_{a+H}$ denotes the indicator function of some coset $a+H$ of a subgroup $H$ of $G$.
This uncertainty principle has an interesting history. The principle, in a more
general form for locally compact abelian (LCA) groups, seems to have
been discovered first by Matolcsi and Sz{\" u}cs \cite{mat} in 1973; the case of equality
was handled by K.T.~Smith \cite{smi} in 1990.
In the case where the group is cyclic an elementary proof was given in 1989 by Donoho
and Stark \cite{ds}, essentially using the BCH bound. They also treat the more general
case where $f$ is {\em highly
concentrated\/} on a subset of the group. Similar investigations can already be found
in the work of Slepian in \cite{sle}.
For further work on uncertainty relations, see for example \cite{wolf1},
\cite{wolf2}, \cite{fol}.

A still somewhat complicated elementary induction proof for finite abelian groups was
given in \cite{pre2}, see also \cite{pre-book}, \cite{pre1}. Their proof uses the
Donoho-Stark principle for cyclic groups that is proved in~\cite{ds}, essentially by
using the BCH bound.
More recently, the principle has been recovered in \cite{qpv}, where a simpler
elementary proof was given. Basically the same proof has been given by Tao
in~\cite{tao}. In that paper, an interesting sharpening of this inequality has been
obtained:
if $G$ is cyclic of prime order $p$,  and $f: G\rightarrow \bC$ is a nonzero function, then
\[|\supp(f)|+|\supp(\hf)| \geq p+1.\]
The proof depends on an old result of Chebotar{\"e}v (see, e.g. \cite{stel}) that is not always valid for finite
fields, see \cite{kra-rep, kra-2008}.
This work has been generalized to all abelian groups by Meshulam in \cite{mes}. Tao later observed that the generalization signifies that for an abelian group $(G,+)$, the points $(\supp(f), \supp(\hf))$ are in the convex hull of the points $(|H|, |G/H|)$ for subgroups $H$ of $G$, see for example \cite{kra-rep, kra-2008}.

Some different type of generalizations are discussed in the next section.

Here it is our aim to present several generalizations of the Donoho-Stark uncertainty principle for nonzero
functions $f: G\rightarrow \bF$, for {\em any\/} field $\bF$ of characteristic zero or
characteristic $p$ not dividing $|G|$. As an aid, we show
that for $\cZ\subseteq \Ghat$
the shift bound $\gd(\Ghat, \cZ)$ satisfies
\[ \gd(\Ghat, \cZ) \geq |\Ghat|/(|\Ghat|-|\cZ|), \]
and we determine when equality holds.
We start with the following.
\begin{teor}\label{shift-lb}
Let $G$ be a finite abelian group with $\Ghat$ its group of $\bE$-characters,
let $\Hhat$ be a subgroup of $\Ghat$, 
and let
$\cZ$ be a proper subset of $\Hhat$.
Write $\cN=\Hhat\setminus\cZ$. 
Then the following hold.\\
(i) There exists a $\cZ$-independent subset of $\Hhat$ of size
at least~$|\Hhat|/|\cN|$, that is, $\gd(\Hhat, \cZ) \geq
|\Hhat|/(|\Hhat|-|\cZ|)$. \\
(ii) If $|\cN|$ divides $|\Hhat |$ and the maximum size of a $\cZ$-independent subset
of $\Hhat$ is $|\Hhat|/|\cN|$, then $\cN$ is a coset of a subgroup of $\Hhat$.
\end{teor}
\begin{proof}
(i) We claim that if $\cA\subseteq \Hhat$ and $|\cA|<|\Hhat|/|\cN|$, then there exists some $\psi\in \Hhat$ such that 
$\cA-\psi=\{\alpha-\psi\mid \alpha\in\cA\}\subseteq \cZ$. Indeed, for all $\eta\in\cN$ we have 
$\eta\in\cA - \psi$ if and only if $\psi\in \cA -\eta;$ hence
$$\cA-\psi\subseteq\cZ\; {\rm if\;
and\; only\; if}\; \psi\notin\cup_{\eta\in\cN} (\cA-\eta).$$
Now if $|\cA|<|\Hhat|/|\cN|$, then
$$|\cup_{\eta\in\cN} (\cA-\eta)|\leq \sum_{\eta\in\cN} |\cA-\eta|\leq
|\cN||\cA|<|\Hhat|;$$
hence there exists a $\psi\in \Hhat$ such that $\psi\not\in\cup_{\eta\in\cN} (\cA-\eta)$, and the claim follows. Since $\cN$ is nonempty by assumption,
part~(i) of the lemma follows by induction from parts 2~and~3 of Lemma~\ref{Zindep}.

(ii) Now suppose that $|\cN|$ divides $|\Hhat |$, that $\cA\subseteq \Hhat$ is
$\cZ$-independent with $|\cA|=|\Hhat|/|\cN|$,
and that no $\cZ$-independent set in $\Hhat$ is larger than $|\Hhat|/|\cN|$.
Since $\cA$ is
$\cZ$-independent, there exists a $\cZ$-independent set $\cB \subseteq \cZ$ and some
$\eta_0\in\cN$ such that $\cA=(\cB\cup\{\eta_0\})-\psi$ for some $\psi\in\Hhat$.
Consider the $\cZ$-independent sets in $\Hhat$ of the form
$\cA_{\eta}=\cB\cup\{\eta\}$ for $\eta\in\cN$.
Each of these sets has the maximum size $|\Hhat|/|\cN|$,
so by our assumptions none of these sets can be shifted inside $\cZ$.
Hence from the analysis in part~(i), we see that for each $\eta\in\cN$ we have that
\[ \Hhat=\displaystyle{\cup_{\eta'\in\cN}} (\cB\cup\{\eta\})-\eta'. \]
Hence if $\Khat=\Hhat\setminus (\cup_{\eta'\in\cN} \cB -\eta')$, then $|\Khat|=|\cN|$
and $\eta-\cN=\Khat$ for all $\eta\in\cN$.
In particular, $\cN-\cN=\Khat$.
So we immediately have that
$\Khat-\Khat=( \eta-\cN)-( \eta-\cN) = \cN-\cN=\Khat$, that is,
$\Khat$ is a subgroup of $\Hhat$, and $\cN=\eta-\Khat$ is a coset of $\Khat$.
\end{proof}


As an immediate consequence we obtain the following generalization of the Donoho-Stark
uncertainty principle.
\begin{teor}[Generalized Donoho-Stark]\label{tmain}
Let $(G,+)$ be a finite abelian group and let $\bF$ be a field of characteristic zero
or characteristic $p$ not dividing $|G|$. Let $\bE$ denote the extension of
$\bF$ containing a primitive $N$-th root of unity, where $N=\expo(G)$,
and let $\Ghat$ denote the group of $\bE$-valued characters of $G$.
Then for any subgroup $\Hhat$ of $\Ghat$ and for
any function $f: G\rightarrow \bF$ 
with Fourier Transform $\hf$ that is nonzero
on~$\Hhat$ we have that
\[ |\supp(f)| | \supp(\hf)\cap \Hhat| \geq |\Hhat|. \]
Equality holds only if $\supp(\hf)\cap \Hhat$ is a coset of a subgroup of $\Hhat$.
In particular, if $f: G\rightarrow \bF$ is nonzero, then
\[ |\supp(f)| | \supp(\hf)| \geq |G|, \]
with equality if and only if $f$ is of the form
\[ f = \gl \chi I_{a+H},\]
for some $\gl \in \bE\setminus\{0\}$, some $\chi\in\Ghat$, some $a\in G$,
and some subgroup~$H$ of~$G$;
here $I_{a+H}$ denotes the indicator function of the coset $a+H$.
\end{teor}
\begin{proof}
If $\cN=\supp(\hf)\cap \Hhat$ is nonempty and if $\cZ=\cZ(f)\cap \Hhat$ is the
collection of zeros of $f$ on $\Hhat$, then by part (i) of Theorem~\ref{shift-lb}
there is a $\cZ$-independent set in $\Hhat$ of size at least $|\Hhat|/|\cN|$;
hence from Theorem~\ref{shift} we conclude that
$|\supp(f)|\geq |\Hhat|/|\cN|$.

Furthermore, part (ii) of Theorem~\ref{shift-lb} shows that if this bound cannot be
improved, then the support $\cN$ of $\fh$ on $\Hhat$ is a coset of a subgroup of
$\Hhat$. Now the last part of Theorem~\ref{tmain} follows by
applying Theorem~\ref{tsup} with $\Ghat=\Hhat$.
\end{proof}

\section{\label{fext}A sharpening of the  Donoho-Stark uncertainty principle}
Our aim in this section is to obtain a sharpening of the Donoho-Stark uncertainty principle by the shifting technique. We need some preparations. Let $G$ be a finite abelian group, and let $S$ be a subset of $G$. We define the 
{\em stabilizer\/} $H(S)$ of $S$ by $H(S)=\{g\in G\mid g+S=S\}$. A few simple properties of the stabilizers are given below.
\ble{LLstab} With the above notation, we have the following:\\
(i) $S$ is a union of cosets of $H(S)$. So $|H(S)|$ divides $|S|$, and in particular, $|H(S)|\leq |S|$.\\
(ii) If $S$ is a coset of a subgroup $K$ of $G$, then $H(S)=K$.\\
(iii) For $S\subseteq G$, we have $H(S)=H(G\setminus S)$. \\
(iv) For $S\subseteq G$, we have $H(S)=\cap_{s\in S} (s-S)$.\\
(v) Let $A$ and $S$ be subsets of $G$ such that $A\cap S=\emptyset$.  Then $(A-S)\cap H(S)=\emptyset$.
\ele
\bpf
(i)  Write $H$ for $H(S)$. Since $s+H\subseteq S$ for every $s\in S$, the set $S$ is a union of cosets of $H$. So $|H(S)|$ divides $|S|$; in particular, $|H(S)|\leq |S|$,  equality holds if and only if $S$ is a coset of $H$.

(ii) Let $S=a+K$ for a subgroup $K$ of $G$. Clearly we have $K\subseteq H(S)$. On the other hand, since $H(S)\subseteq S-S$, we see that $H(S)\subseteq K$. Hence $H(S)=K$.

(iii) Evident.

(iv) Note that $g\in s-S$ for all $s\in S$ precisely when $s-g\in S$ for all $s$, that is, when $S-g\subseteq S$, i.e., when $g\in H(S)$.

(v) If $a-s\in H(S)$ with $a\in A$ and $s\in S$, then $a\in s+H(S)\subseteq S$.
\epf
We are now ready to prove the following improvement of the Donoho-Stark uncertainty principle. 
\begin{teor}[Sharpened Donoho-Stark]\label{LTmaini}
Let $(G,+)$ be a finite abelian group and let $\bF$ be a field of characteristic zero
or characteristic $p$ not dividing $|G|$. Let $\bE$ denote an extension of
$\bF$ containing a primitive $N$-th root of unity, where $N=\expo(G)$,
and let $\Ghat$ denote the group of $\bE$-valued characters of $G$.
Then for
any nonzero function $f: G\rightarrow \bF$
with Fourier transform $\hf$,
we have that
\beql{LEtmain11} |\supp(f)| | \supp(\hf)| \geq |G|+|\supp(\hf)|-|H(\supp(\hf)|;\eeql
and dually,
\beql{LEtmain12} |\supp(f)| | \supp(\hf)| \geq |G|+|\supp(f)|-|H(\supp(f)|.\eeql
\end{teor}
\bpf
Let $\cZ\subseteq \Ghat$ and write $\cN=\Ghat\setminus \cZ$.  Assume that $\cN\neq \emptyset$. Suppose that $\cA$ is a $\cZ$-independent set of maximum size. Then $\cA$ is of the form $\cA=\cB\cup\{\eta_0\}-\chi_0$ with $\cB\subseteq \cZ$ being $\cZ$-independent and $\eta_0\in \cN$, and by the assumption that $\cA$ has maximum size, no $\cZ$-independent set $\cB\cup\{\eta\}$ (with $\eta\in \cN$) can be extended any further; that is, for each $\chi\in\Ghat$ and each $\eta\in \cN$, we have $\cB\cup\{\eta\} -\chi\not\subseteq \cZ$. So for every $\eta\in \cN$ and $\chi\in \Ghat$, there exists $\eta'\in \cN$ such that $\eta'\in\cB\cup\{\eta\}-\chi$; that is, $\chi\in \cB\cup\{\eta\} -\cN$. Hence if $\chi\notin \cB-\cN$, then $\chi\in \eta-\cN$; since this holds for every $\eta\in \cN$, we have $\chi\in \cap_{\eta\in \cN}(\eta-\cN)=H(\cN)$, where the equality follows from \Le{LLstab}, part (iv). Since $\cB\subseteq \cZ=\Ghat\setminus \cN$ by assumption, we have $\cB\cap \cN=\emptyset$; hence according to  \Le{LLstab}, part (v), we have $(\cB-\cN)\cap H(\cN)=\emptyset$, and it now follows that
\[\cB-\cN=\Ghat\setminus H(\cN).\]
Using $|\cB-\cN|\leq |\cB||\cN|$ and $|G|=|\Ghat|$, we  conclude that
\beql{LEAN1} |G|-|H(\cN)| =|\cB-\cN|\leq |\cB||\cN|.\eeql
Now suppose that $f: G\rightarrow \bF$ is nonzero, and let $\cN=\supp(\hf)$ and $\cZ=\cZ(f)$. Suppose that $\cA$ is a $\cZ$-independent set of maximum size. By the shifting bound, we have $|\supp(f)|\geq |\cA|$. As we have shown above, we can write $\cA=\cB\cup\{\eta_0\}-\chi_0$ with $\cB\subseteq \cZ$ being $\cZ$-independent and $\eta_0\in \cN$, and $\cB-\cN=\Ghat\setminus H(\cN)$.
It follows that
\[|\supp(f)||\supp(\hf)|\geq (|\cB|+1)|\cN| \geq |G|-|H(\cN)|+|\cN|=|G|-|H(\supp(\hf))|+|\supp(\hf)|.\]

Since $(\hf)^*=f$ and $|G|=|\Ghat|$, the second inequality in the theorem follows by dualizing (i.e., replacing $G$ by $\Ghat$ and interchanging $f$ and $\hf$).
\epf

\bre{LRtmain1} \rm Note that by \Le{LLstab}, part (i), we have $|H(\supp(\hf))|\leq |\supp(\hf)|$. It follows that the right hand side of (\ref{LEtmain11}) is greater than or equal to $|G|$. So (\ref{LEtmain11}) is a sharpening of the Donoho-Stark uncertainty principle. When $\cN=\supp(\hf)$ is a coset of $H(\cN)$, then $H(\cN)=\cN$ and the inequality (\ref{LEtmain11}) reduces to the Donoho-Stark inequality. We also remark that the lower bound on $|{\rm supp}(f)|$ arising from (\ref{LEtmain11}) improves the bound obtained from the Donoho-Stark inequality provided that $|G|\pmod {|{\rm supp}(\hf)|}$, the least non-negative remainder, is greater than $|H({\rm supp}(\hf)|$. For an example of this situation, see Example~\ref{LEham}.
\ere
\bex{LEham}\rm
Let $n=2^d-1$, $G=\bZ_n$, and $\Ghat =\{1,\ga,\ldots ,\ga^{n-1}\}$, where $\ga$ is a primitive element of $\bF_{2^d}$. Let $\cZ=\{\ga,\ga^2,\ga^{2^2}, \ldots, \ga^{2^{d-1}}\}\subseteq \Ghat$ and $\cN=\Ghat\setminus \cZ$. Then $H(\cN)=H(\cZ)=\{1\}$. Let $f: \bZ_n \rightarrow \bF_2$ be a nonzero function with $f(\ga)=0$. Note that the condition on $f$ implies that the associated codeword $(f_0, f_1, \ldots, f_{n-1})$ is contained in the $[n=2^d-1, k=2^d-1-d, 3]$ binary Hamming code. Now if $f\neq 0$ and $f$ has no additional zeros, then for $d\geq 3$  the Donoho-Stark bound for $f$ gives
\[w(f)\geq \left\lceil \frac{2^d-1}{2^d-1-d}\right\rceil =1+ \left\lceil \frac{d}{2^d-1-d}\right\rceil=2,\]
while our improved bound gives
\[w(f)\geq   \left\lceil \frac{2^d-1+2^d-1-d-1}{2^d-1-d}\right\rceil =2+ \left\lceil \frac{d-1}{2^d-1-d}\right\rceil=3,\]
showing that our new bound can improve the weight estimate afforded by the Donoho-Stark bound.
\eex
In the remainder of this section, we investigate the case of equality in (\ref{LEtmain11}). Note that if we have equality in the Donoho-Stark inequality, then $\supp(f)$ is a coset of a subgroup $K$ of $G$, hence $|H(\supp(f)|=|K|=|\supp(f)|$ and we also have equality in the sharpened versions (\ref{LEtmain11}) and (\ref{LEtmain12}) of that inequality. We will call this the {\em classical\/} case of equality.  Note that when one of $|H(\supp(f))|=|\supp(f)|$ or $|H(\supp(\hf)|=|\supp(\hf)|$ holds, they both hold and then (and only then)  we are in the classical case. We next describe a simple non-classical example.
\bex{LEs2}\rm
Let $(G,+)$ be a finite abelian group with identity 0, and let $\bF$ be a field of characteristic zero
or characteristic $p$ not dividing $|G|$. Let $a\in G\setminus \{0\}$, and define $f: G\rightarrow \bF$ by letting $f(0)=1$, $f(a)=-1$, and $f(g)=0$ for $g\neq 0,a$. We claim that for this $f$ equality in~(\ref{LEtmain11}) holds.
Obviously, $\supp(f)=\{0,a\}$. Furthermore, for $\chi\in \Ghat$, we have that $\hf(\chi)=\sum_{x\in G} f(x)\chi(-x)=1-\chi(a)$; hence $\chi\in \supp(\hf)$ iff $\chi(a)\neq 1$, i.e., iff $\chi\notin \langle a\rangle^\perp$; so $\supp(\hf)=\Ghat\setminus \langle a\rangle^\perp$. 
Finally, $H(\supp(\hf))=H(\Ghat\setminus \langle a\rangle^\perp)=H( \langle a\rangle^\perp)= \langle a\rangle^\perp$ by \Le{LLstab}, parts (iii) and (ii). 
Using the above, we find that $|\supp(f)||\supp(\hf)|=|G|-|H(\supp(\hf))|+|\supp(\hf)|=2(|G|-|\langle a\rangle^\perp|)$, so we indeed have equality in (\ref{LEtmain11}). This example is a classical one if and only if $|\Ghat|=2| \langle a\rangle^\perp|$.  Since $a\neq 0$, we have  $\langle a\rangle^\perp\neq \Ghat$; so using  (\ref{LEorth}), we see that this example is a classical one if and only if $|\langle a\rangle|=2$, that is, if and only if $\langle a\rangle=\{0,a\}$ and $2a=0$.

Let us now investigate when we also have equality in (\ref{LEtmain12}).
Using (\ref{LEorth}), we have that $|\supp(f)||\supp(\hf)|=2(|G|-|G|/|\langle a\rangle|)$, and $|G|+|\supp(f)|-|H(\supp(f)|=|G|+2-|H\{0,a\})|$.
We now distinguish two cases. First, if $H(\{0,a\})=\{0,a\}$, that is, if  $2a=0$, then $\{0,a\}=\langle a\rangle$ and we always have equality in (\ref{LEtmain12}); this is a classical example. Second, if $H(\{0,a\})=\{0\}$, that is, if $2a\neq 0$, then $|\langle a\rangle|>2$ and we have equality precisely when $G=\langle a\rangle$ with $|G|=3$; we can take $G=\bZ_3$ and $a=1$. This is again a non-classical example.
\eex

In order to have equality in  (\ref{LEtmain12}), we need sets $B, N\subseteq G$ with $B\cap N=\emptyset$ satisfying
\beql{LEnf} G-H(N)=B-N, \qquad |B-N|=|B||N|,\eeql
where $N$ is a union of cosets of $H(N)$.
(Note that here, for convenience, we have dualized.) We ask whether from such sets, we can construct a function $f$ for which $\supp(f)=N$ and $|supp(\hf)|=|B|+1$? Such sets give rise to what is called a {\em near-factorization\/} \cite[Section 9.3]{ss}, \cite{cghk}, \cite{pech}, \cite{ss-pap}, \cite{bhs}. Indeed, put $G_0=G/H(N)$; since $N$ is a union of cosets of $H(N)$, we have that $N=H(N)-D$ for some set $D$ of size $|N|/H(N)|$. Then $G_0\setminus\{0\}=B_0+D$ (direct sum), where $B_0=\{b+H(N)\mid b\in B\}$. It is conjectured that near-factorizations of abelian groups exist only for cyclic groups \cite{ss}.
\bex{LEnf1}\rm
Let $G=\bZ_n$ with $n-1=uv$. Take $N=\{0, 1, \ldots, u-1\}$ and $A=\{u,2u, \ldots, (v-1)u\}$. Then $A-N=G\setminus\{0\}$, so $(N,-A)$ is a near-factorization of~$G$.
\eex
For other examples of near-factorizations and further discussions, we refer to the references given above.

\section{\label{con} Conclusion}
In this paper, we present the shift bound for abelian codes. We give two proofs, one by using the approach developed in \cite{lw}, and the other by using the method from \cite{mes-pq}. We use the shift bound for abelian codes to prove a generalization (Theorem~\ref{tmain}) of the Donoho-Stark uncertainty principle. Furthermore, a pair of inequalities stronger than the Donoho-Stark uncertainty principle is proved by using the shifting technique. While the equality case in the Donoho-Stark uncertainty principle can be characterized completely, it seems not easy to characterize the equality case in these new inequalities. We leave this as a problem for further research.

\vspace{0.1in}
\noindent{\bf Acknowledgements.} The work of Tao Feng is supported by Natural Science Foundation of China (Grant No. 11771392). The work of Qing Xiang is partially supported by an NSF grant DMS-1600850. The third author, Qing Xiang, would like to thank Zhejiang University and the Three-Gorge Math Research Center (TGMRC), where this work was partially carried out.


\begin{thebibliography}{99}
%
\bibitem{bhs}G.~Bacs\'o, L.~H\'ethelyi, and P.~Sziklai, {\em New near-factorizations of finite groups\/},  Studia Scientiarum Mathematicarum Hungarica 45.4 (2008): 493-510.
%
\bibitem{ber1} S.D.~Berman, {\em On the theory of group codes\/}, Kibernetika, 3:
31--39, 1967.
%
\bibitem{ber2} S.D.~Berman, {\em Semisimple cyclic and abelian codes II\/},
Kibernetika, 3: 21--30, 1967.
%
\bibitem{bla} R.E.~Blahut, Theory and practice of error control codes, Reading, MA:
Addison-Wesley 1983.
%
\bibitem{cghk}D.~DeCaen, D., D.A.~Gregory, I.G.~Hughes and D.L.~Kreher, {\em Near-factors of finite groups\/}, Ars Combinatoria, 29, 1990, pp.\ 53--63.
%
\bibitem{cam} P.~Camion, Abelian codes, Technical report 1059, Mathematical Research
Center, University of Wisconsin, 1971.
%
\bibitem{de} D.L.~Donoho and M.~Elad, {\em Optimally sparse representation in
general (nonorthogonal) dictionaries via $\ell^1$ minimization\/}, Proc.\ Nat.\
Acad.\ Sci., vol.\ 100, no.\ 5, Mar.\ 2003, pp.\ 2197--2202.
%
\bibitem{dh}
D.L.~Donoho and X.~Huo, {\em Uncertainty principles and ideal atomic decomposition\/},
IEEE Trans.\ on Inform.\ Theory, vol.\ 47, no.\ 7, Nov.\ 2001, pp.\ 2845--2862.
%
\bibitem{ds}
D.L.~Donoho and P.B.~Stark, {\em Uncertainty principles and signal recovery\/},
SIAM J.\ Applied Math.\ {\bf 49} (1989) 906--931.
%
\bibitem{dgv}
V.~Diaz, C.~Guevara, M.~Vath, {\em Codes from $n$-Dimensional Polyhedra and
$n$-Dimensional Cyclic Codes\/}, Proceedings of SIMU summer institute, 2001.
%
\bibitem{el-book}M.~Elad, Sparse and Redundant Representations: From Theory to Applications in Signal and Image Processing, Springer Science \& Business Media, 2010.
%
\bibitem{eb} M.~Elad and A.M.~Bruckstein, {\em A generalized uncertainty principle and
sparse representation in pairs of bases\/}, IEEE Trans.\ on Inform.\ Theory, vol.\ 48,
no.\ 9, Sept.\ 2002, pp.\ 2558--2567.
%
\bibitem{fol}
G.B.~Folland, A.~Sitaram,
{\em The uncertainty principle: a mathematical survey\/},
J.\ Fourier Anal.\ Appl.\ 3 (1997), no.\ 3, 207--238.
%
\bibitem{fre}P.~E.~Frenkel, {\em Simple proof of Chebotar\"ev's theorem on roots of unity}, arXiv preprint math/0312398 (2003).
%
\bibitem{for} G.\ David Forney, Jr., {\em Transforms and Groups\/}, in Codes, Curves
and Signals: Common Threads in Communications, A.~Vardy (Ed.), Norwood, MA: Kluwer,
1998, Chapter 7.
%
\bibitem{ga}W.~Gharibi, {\em An improved lower bound of the spark with applications\/}, Computer Science 2012.
%
\bibitem{gn} R.~Gribonval and M.~Nielsen, {\em Sparse representations in unions
of bases\/}, IEEE Trans.\ on Inform.\ Theory, vol.\ 49, no.\ 12, Dec.\ 2003,
pp.\ 3320--3325.
%
\bibitem{hew} E.~Hewitt and K.A.~Ross, Abstract harmonic analysis I, New York:
Springer 1979.
%
\bibitem{kra-rep}
F.~Krahmer, G.E.~Pfander and P.~Rashkov, {\em Support size conditions for time-frequency representations on finite abelian groups}, Tech.\ Rep.\ 13, School of Engineering and Science, Jacobs University, 2007.
%
\bibitem{kra-2008} F.~Krahmer, G.E.~Pfander and P.~Rashkov,
{\em Uncertainty in time-frequency representations on finite abelian groups and applications\/},
Appl.\ Comput.\ Harmon.\ Anal.\ 25 (2008) 209–225.
%
\bibitem{kupp}P.~Kuppinger, G.~Durisi and H.~Bolcskei, {\em Uncertainty Relations and Sparse Signal Recovery for Pairs of General Signal Sets\/}, IEEE Trans.\ on Inform.\ Theory, vol.\ 58, no.\ 1, pp.\ 263--277, Jan.\ 2012.
%
\bibitem{lw} J.~H.\ van Lint, R.~M.\ Wilson, {\em On the minimum distance of
cyclic codes\/}, IEEE Trans.\ Inform.\ Theory, {\bf 32} (1986), 23--40.
%
\bibitem{mcws} F.J.~MacWilliams and N.J.A.~Sloane, The theory of error-correcting
codes, Amsterdam: North-Holland 1977.
%
\bibitem{man} H.B.~Mann, Addition theorems, Wiley\&Sons, 1965.
%
\bibitem{mat} T.~Matolcsi and J.~Sz\"ucs, {\em Intersection des mesures spectrales
conjug\'ees\/}, C.\ R.\ Acad.\ Sci\ Paris {\bf 277} (1973), 841--843.
%
\bibitem{pre2}
E.\ Matusiak, M.~\"Ozaydin, T.~Przebinda, {\em The Donoho-Stark uncertainty principle
for a finite abelian group\/},
Acta Math.\ Univ.\ Comenianae, vol.\ LXXIII, 2 (2004), pp.\ 155--160.
%
\bibitem{mes} Roy Meshulam, {\em An uncertainty inequality for finite abelian
groups\/}, Europ.\ J.\ of Combinatorics 27 (2006) 63–67.
%
\bibitem{mes-pq}Roy Meshulam, {\em An uncertainty inequality for groups of order $pq$\/}, Europ.\ J.\ Combinatorics (1992) 13, 401-407.
%
\bibitem{pech}A.~P\^echer, {\em Cayley partitionable graphs and near-factorizations of finite groups\/},  Discrete Math.\ vol.\ 276, no.\ 1 (2004): 295-311.
%
\bibitem{pre1}
M.~\"Ozaydin and T.~Przebinda, {\em An entropy-based uncertainty principle for a
locally compact abelian group\/}, J.\ of Functional Analysis 215 (2004) 241--252.
%
\bibitem{pre-book}
T.~Przebinda, {\em Three uncertainty principles for a locally compact abelian
group\/}, in: Representations of real and $p$-adic groups, Eng-Chye Tan and Chen-Bo
Zhu (Ed.), Lecture Notes Series, Institute for Mathematical Sciences, National
University of Singapore, Vol.\ 2., April 2004.
%
\bibitem{qpv} M.~Quisquater, B.~Preneel, J.~Vandewalle, {\em A new inequality
in Discrete Fourier theory}, IEEE Trans.\ on Inform.\ Theory, vol.\ 49, no.\ 8,
August 2003, pp.\ 2038--2040.
%
\bibitem{ri14}B.~Ricaud, B.~Torr\'esani, {\em A survey of uncertainty principles and some signal processing applications\/},
Adv.\ in Comp.\ Math., Vol.\ 40, nr.\ 3, 2014, pp.\ 629-650.
%
\bibitem{ri13}B.~Ricaud and B.~Torr\'esani, {\em Refined Support and Entropic Uncertainty Inequalities\/}, IEEE Trans.\ on Inform.\ Theory, vol.\ 59, no.\ 7, pp.\ 4272--4279, July 2013.
%
\bibitem{rud} W.~Rudin, Fourier analysis on groups, New York: Wiley 1990.
%
\bibitem{ss-pap}T.~Sakuma and H.~Shinohara. {\em Krasner near-factorizations and 1-overlapped factorizations\/}.The Seventh European Conference on Combinatorics, Graph Theory and Applications. Edizioni della Normale, Pisa, 2013.
%
\bibitem{sle} D.~Slepian, {\em Prolate spheroidal wave functions, Fourier analysis,
and uncertainty - V: The discrete case\/}, Bell Syst.\ Tech.\ J., vol.\
57, no.\ 5, May-June 1978, pp.\ 1371--1430.
%
\bibitem{smi} K.T.~Smith, {\em The uncertainty principle on groups\/}, SIAM J.\ Appl.\
Math.\ {\bf 50} (1990), 876--882.
%
\bibitem{stel}
P.~Stevenhagen, H.W.~Lenstra Jr., {\em Chebotar\"ev and his density theorem\/},
Math.\ Intelligencer {\bf 18} (1996), no.\ 2, 26--37.
%
\bibitem{ss}S.~Szab\'o and A.D.~Sands. Factoring groups into subsets. CRC Press, 2009.
%
\bibitem{tao} T.~Tao, {\em An uncertainty principle for cyclic groups of prime
order\/}, Mathematical Research Letters {\bf 12} (2005) 121--127.
%
\bibitem{ter}
A.~Terras, Fourier Analysis on Finite Groups and Applications, Cambridge University
Press, Cambridge 1999.
%
%
%
\bibitem{wolf1}
J.A.~Wolf, 
{\em The uncertainty principle for Gel\cprime fand pairs\/},
Nova J.\ Algebra Geom.\ 1 (1992), no.\ 4, 383--396.
%
\bibitem{wolf2}
J.A.~Wolf,
{\em Uncertainty principles for Gel\cprime fand pairs and Cayley complexes\/},
75 years of Radon transform (Vienna, 1992), 271--292,
Conf.\ Proc.\ Lecture Notes Math.\ Phys., IV,
Internat.\ Press, Cambridge, MA, 1994.
%
\end{thebibliography}
\end{document}